\title{Fast Strategies In Maker-Breaker Games Played on Random Boards}
\author{
\quad{Dennis Clemens
\thanks{Department of Mathematics and Computer Science, Freie Universit\"{a}t Berlin, Germany. Email: d.clemens@fu-berlin.de.
Research supported by DFG, project SZ 261/1-1.}}
\quad{Asaf Ferber
\thanks{School of Mathematical Sciences, Raymond and Beverly Sackler Faculty of Exact Sciences, Tel Aviv University, Tel Aviv, 69978, Israel. Email: ferberas@post.tau.ac.il}}
\quad{Michael Krivelevich
\thanks{School of Mathematical Sciences,
Raymond and Beverly Sackler Faculty of Exact Sciences, Tel Aviv
University, Tel Aviv, 69978, Israel. Email: krivelev@post.tau.ac.il.
Research supported in part by USA-Israel BSF grant 2006322, by grant
1063/08 from the Israel Science Foundation, and by a Pazy memorial
award.}} \quad{Anita Liebenau
\thanks{Department of Mathematics and Computer Science, Freie Universit\"{a}t Berlin, Germany. Email: liebenau@math.fu-berlin.de.
Research supported by DFG within the graduate school Berlin Mathematical School.}}}

\documentclass[11pt,a4paper]{article}
\usepackage{amsmath,amssymb,latexsym,color,epsfig,a4, enumerate}
\usepackage{bbm}
\usepackage{a4wide} 

\newif\ifnotesw\noteswtrue

\DeclareFontFamily{OT1}{pzc}{}
\DeclareFontShape{OT1}{pzc}{m}{it}{<-> s * [1.10] pzcmi7t}{}
\DeclareMathAlphabet{\mathpzc}{OT1}{pzc}{m}{it}

\def\({\left(}
\def\){\right)}

\def\d{\delta}

\def\Prob{{\bf Pr}}

\def\cf{{\cal F}}

\newtheorem{theorem}{Theorem}[section]
\newtheorem{lemma}[theorem]{Lemma}
\newtheorem{claim}[theorem]{Claim}

\newtheorem{definition}[theorem]{Definition}

\numberwithin{equation}{section}

\newcommand{\reals}{\ensuremath{\mathbbm{R}}}

\newcommand{\Exp}{\ensuremath{\mathbbm{E}}}

\def\cG{{\cal G}}

\renewcommand{\epsilon}{\varepsilon}
\newcommand{\gnp}{\ensuremath{\mathcal{G}_{n,p}}}
\newcommand{\nat}{\ensuremath{\mathbbm{N}}}
\newcommand{\Deg}{\ensuremath{\mathpzc{Deg}}}

\newenvironment{proof}{\noindent{\bf Proof\,}}{\hfill$\Box$}

\begin{document}
\maketitle

\begin{abstract}
In this paper we analyze classical Maker-Breaker games played on the
edge set of a sparse random board $G\sim \gnp$. We consider the
Hamiltonicity game, the perfect matching game and the
$k$-connectivity game. We prove that for $p(n)\geq \text{polylog}(n)/n$,
the board $G\sim \gnp$ is typically such that Maker can win these
games asymptotically as fast as possible, i.e. within $n+o(n)$,
$n/2+o(n)$ and $kn/2+o(n)$ moves respectively.
\end{abstract}

\section{Introduction} \label{sec::intro}

%
%
%

Let $X$ be any finite set and let ${\mathcal F}\subseteq 2^X$ be a
family of subsets. Usually, $X$ is called the {\em board}, whereas
${\mathcal F}$ is referred to as the family of {\em winning sets}.
In the $(a:b)$ \emph{Maker-Breaker} game $(X,{\mathcal F})$ (also
known as a \emph{weak game}), two players called Maker and Breaker
play in rounds. In every round Maker claims $a$ previously unclaimed
elements of the board $X$ and Breaker responds by claiming $b$
previously unclaimed elements of the board. Maker wins as soon as he
fully claims all elements of some $F\in {\mathcal F}$. If Maker does
not fully claim any winning set by the time all board elements are
claimed, then Breaker wins the game. The most basic case is $a=b=1$,
the so called \emph{unbiased} game. Notice that being the first
player is never a disadvantage in a Maker-Breaker game. Therefore,
in order to prove that Maker can win some Maker-Breaker game as the
first or the second player it is enough to prove that he can win
this game as a second player. Hence, we will always assume that
Maker is the second player to move.

It is natural to play Maker-Breaker games on the edge set of a graph
$G=(V,E)$ with $|V|=n$. In this case, $X=E$. In the {\em
connectivity game}, Maker wins if and only if his edges contain a
spanning tree. In the \emph{perfect matching} game ${\mathcal
M}_n(G)$ the winning sets are all sets of $\lfloor n/2 \rfloor$
independent edges of $G$. Note that if $n$ is odd, then such a
matching covers all vertices of $G$ but one. In the Hamiltonicity
game ${\mathcal H}_n(G)$ the winning sets are all edge sets of
Hamilton cycles of $G$. Given a positive integer $k$, in the
$k$-connectivity game ${\mathcal C}_n^k(G)$ the winning sets are all
edge sets of $k$-connected subgraphs of $G$.

Maker-Breaker games played on the edge set of the complete graph
$K_n$ are well studied.
In this case, it is easy to see
(and also follows from \cite{L}) that for every $n\geq 4$, Maker can
win the unbiased connectivity game in $n-1$ moves, which is clearly
the best possible. It was proved in \cite{HKSS} that Maker can win
the unbiased perfect matching game on $K_n$ within $n/2+1$ moves (which is
clearly the best possible), the unbiased Hamiltonicity game within
$n+2$ moves and the unbiased $k$-connectivity game within
$kn/2+o(n)$ moves.

In \cite{HS}, it was shown that Maker can win the unbiased
Hamiltonicity game on $K_n$ within $n+1$ moves which is clearly the best
possible and recently it was proved (see \cite{FH}) that Maker can
win the unbiased $k$-connectivity game within $kn/2+1$ moves which
is clearly the best possible.

It follows from all these results that many natural games played on
the edge set of the complete graph $K_n$ are drastically in favor of
Maker. Hence, it is natural to try to make his life a bit harder and
to play on different types of boards or to limit his number of
moves. In this paper we are mainly interested in the following two
questions.

\begin{description}

\item [$(i)$] Given a sparse board $G=(V,E)$, can Maker win the game
played on this board?
\item [$(ii)$] How fast can Maker win this game?
\end{description}

In \cite{SS} it was suggested to play Maker-Breaker games on the
edge set of a random graph $G\sim \gnp$ and some games were examined
such as the perfect matching game, the Hamiltonicity game, the
connectivity game and the $k$-clique game.

Later on, in \cite{BFHK}, it was proved that the edge set of $G\sim
\gnp$ with $p=(1+o(1))\frac{\ln n}{n}$ is typically such that Maker
has a strategy to win the unbiased perfect matching game, the
Hamiltonicity game and the $k$-connectivity game. This is best
possible since $p=\frac{\ln n}{n}$ is the threshold probability for
the property of $\gnp$ having an isolated vertex. Moreover, the
proof in \cite{BFHK} is of a "hitting-time" type. That means, in the
random graph process, i.e. when adding one new edge randomly every
time, typically at the moment the graph reaches the needed minimum
degree for winning the desired game, Maker indeed can win this game.
For example, at the first time the graph process achieves minimum
degree 2 the board is typically such that Maker wins the perfect
matching game.

Another type of games is the following. Let $X$ be any finite set
and let ${\mathcal F}\subseteq 2^X$ be a family of subsets. In the
\emph{strong game} $(X,{\mathcal F})$, two players called \emph{Red}
and \emph{Blue}, take turns in claiming one previously unclaimed
element of $X$, with Red going first. The winner of this game is the
first player to fully claim all the elements of some $F\in {\mathcal
F}$. If no one wins by the time all the elements of $X$ are claimed,
then the game ends in a \emph{draw}. For example, the classic
Tic-Tac-Toe is such a game. It is well known from classic Game
Theory, that for every strong game $(X,{\mathcal F})$, either Red
has a winning strategy or Blue has a drawing strategy. For certain
games, a hypergraph coloring argument can be used to prove that a
draw is impossible and thus these games are won by Red. However,
these arguments are purely existential. That is, even if it is known
that Red has a winning strategy for some strong game $(X,{\mathcal
F})$, it might be very hard to describe such a strategy explicitly.

Using fast strategies for Maker in certain games, explicit
strategies for Red were given for games such as the perfect matching
game, the Hamiltonicity game and the $k$-connectivity game played on
the edge set of $K_n$ (see \cite{FH} and \cite{FH1}). This provides
substantial motivation for studying fast winning strategies in
Maker-Breaker games.

Regarding the strong game played on $G\sim \gnp$, nothing is known
yet. Hence, as a first step for finding explicit strategies for Red
in the strong game played on a random board, it is natural to look
for fast winning strategies for Maker in the analogous games.
Therefore the following question is quite natural.

\begin{description}
\item [$Question(s):$] Given $p=p(n)$, how fast can Maker win the perfect matching, the
Hamiltonicity and the $k$-connectivity games played on the edge set
of a random board $G\sim \gnp$?
\end{description}

In this paper we resolve these questions for a wide range of the
values of $p=p(n)$. We prove the following theorems:

\begin{theorem} \label{PMgame}
Let $b \geq 1$ be an integer, let $K >12$, $p=\frac{\ln^{K} n}{n}$,
and let $G\sim \gnp$. Then a.a.s. $G$ is such that in the $(1:b)$
weak game ${\mathcal M}_n(G)$, Maker has a strategy to win within
$\frac{n}{2}+o(n)$ moves.
\end{theorem}

\begin{theorem} \label{HAMgame}
Let $b \geq 1$ be an integer, let $K >100$, $p=\frac{\ln^{K} n}{n}$,
and let $G\sim \gnp$. Then a.a.s. $G$ is such that in the $(1:b)$
weak game ${\mathcal H}_n(G)$, Maker has a strategy to win within
$n+o(n)$ moves.
\end{theorem}

\begin{theorem} \label{KCONgame}
Let $b\geq 1$, $k\geq 2$ be two integers, let ${K}>100$,
$p=\frac{\ln^K n}{n}$,  and let $G\sim \gnp$. Then a.a.s. $G$ is
such that in the $(1:b)$ weak game ${\mathcal C}_n^k(G)$, Maker has
a strategy to win within $\frac{kn}{2}+o(n)$ moves.
\end{theorem}

Due to obvious monotonicity the results are valid for any $p=p(n)$
larger than stated in the theorems above.

For the sake of simplicity and clarity of presentation, we do not
make a particular effort to optimize the constants obtained in our
proofs. We do not believe that the order of magnitude we assume for
$p$ in the above theorems is optimal. We also omit floor and ceiling
signs whenever these are not crucial. Most of our results are
asymptotic in nature and whenever necessary we assume that $n$ is
sufficiently large.

The remaining part of the paper is organized as follows. First, we
introduce the necessary notation. In Section \ref{sec::tools}, we
assemble several results that we need. We give some basic results of
positional games in Section \ref{subsec::games}, of graph theory in
Section \ref{subsec::graphtheory}, and about $\gnp$ in Section
\ref{subsec::gnp}. The strategy of Maker (in each of the three
games) includes building a suitable expander on a subgraph, which
then contains the desired structure. We therefore include results
about expanders in Section \ref{subsec::expanders}. We prove Theorem
\ref{PMgame}, \ref{HAMgame} and \ref{KCONgame} in Sections
\ref{sec::PMgame}, \ref{sec::HAMgame} and \ref{sec::KCONgame},
respectively. Finally, in Section \ref{sec::OpenProblems} we pose
some open problems connected to our results.
\subsection{Notation and terminology} \label{subsec::prelim}
\noindent Our graph-theoretic notation is standard and follows that
of ~\cite{West}. In particular, we use the following.

For a graph $G$, let $V(G)$ and $E(G)$ denote its sets of vertices
and edges, respectively.
Let $S,T \subseteq V(G)$ be subsets.
Let $G[S]$ denote the subgraph
of $G$, induced on the vertices of $S$, and let
$E(S) = E(G[S])$.
Further, let $E(S,T) := \{ st \in E(G) : s \in S, t \in T \}$,
and let $N(S):= \{ v \in V : \exists s \in S \text{ s.t. } vs \in E(G) \}$
denote the neighborhood of $S$.
For an edge $e \in E(G)$ we
denote by $G - e$ the graph with vertex set $V(G)$ and edge set
$E(G)\setminus \{e\}$.

Assume that some Maker-Breaker game, played on the edge set of some
graph $G$, is in progress. At any given moment during the game, we
denote the graph formed by Maker's edges by $M$, and the graph
formed by Breaker's edges by $B$. At any point during the game, the
edges of $F:= G \setminus (M\cup B)$ are called \emph{free edges}.
For subsets $S,T \subseteq V$, let $N_M(S) := \{ v \in V : \exists s
\in S \text{ s.t. } vs \in E(M) \}$, $E_M(S) := E(S) \cap E(M)$, and
$E_M(S,T) := E(S,T) \cap E(M)$. Further, for $v\in V$, let $d_M
(v,S) = |E_M(\{v\}, S)|$, and $d_M(v) := d_M(v,V)$. Analogously, we
define $N_B(S)$, $N_F(S)$, $E_B(S)$, $E_F(S)$, etc.

\section{Auxiliary results}
\label{sec::tools}

In this section we present some auxiliary results that will be used
throughout the paper.

First,  we will need to employ bounds on large deviations of random
variables. We will mostly use the following well-known bound on the
lower and the upper tails of the Binomial distribution due to
Chernoff (see \cite{AS}, \cite{JLR}).

\begin{lemma}\label{Che}
If $X \sim Bin(n,p)$, then
\begin{itemize}
    \item $\Prob[X<(1-a)np]<exp\left(-\frac{a^2np}{2}\right)$ for every $a>0.$
    \item $\Prob[X>(1+a)np]<exp\left(-\frac{np}{3}\right)$ for every $a\geq 1.$
\end{itemize}
\end{lemma}

\begin{lemma}\label{l:Che}
Let $X \sim Bin(n,p)$, $\mu=\Exp(X)$ and $k \geq 7\mu$, then
$\Prob(X \geq k) \leq e^{-k}$.
\end{lemma}

\subsection{Basic positional games results}
\label{subsec::games}
The following fundamental theorem, due to Beck~\cite{Beck}, is a
useful sufficient condition for Breaker's win in the $(a:b)$ game
$(X, {\mathcal F})$. It will be used extensively throughout the
paper.
\begin{theorem} \label{bwin}
Let $X$ be a finite set and let ${\mathcal F} \subseteq 2^X$. If
$\sum_{F \in {\mathcal F}}(1+b)^{-|F| / a} < \frac{1}{1+b}$, then
Breaker (as the first or second player) has a winning strategy for
the $(a : b)$ game $(X, {\mathcal F})$.
\end{theorem}

While Theorem~\ref{bwin} is useful in proving that Breaker wins a
certain game, it does not show that he wins this game quickly. The
following lemma is helpful in this respect.

\begin{lemma} [Trick of fake moves] \label{lem::fakeMoves}
Let $X$ be a finite set and let ${\mathcal F} \subseteq 2^X$. Let
$b' < b$ be positive integers. If Maker has a winning strategy for
the $(1 : b)$ game $(X, {\mathcal F})$, then he has a strategy to
win the $(1 : b')$ game $(X, {\mathcal F})$ within $1 + |X|/(b+1)$
moves.
\end{lemma}

The main idea of the proof of Lemma~\ref{lem::fakeMoves} is that, in
every move of the $(1 : b')$ game $(X, {\mathcal F})$, Maker (in his
mind) gives Breaker $b - b'$ additional board elements. The
straightforward details can be found in~\cite{BeckBook}.

We will also use a variant of the classical \emph{Box Game}
first introduced by Chv\'atal and Erd\H{o}s in \cite{CE}. The {\em
Box Game with resets $rBox(m,b)$}, first studied in \cite{FHK}, is
played by two players, called BoxMaker and BoxBreaker. They play on
a hypergraph ${\mathcal H} = \{A_1, \ldots, A_m\}$, where the sets
$A_i$ are pairwise disjoint. BoxMaker claims $q$ elements of
$\bigcup_{i=1}^m A_i$ per turn, and then BoxBreaker responds by {\em
resetting} one of BoxMaker's \emph{boxes}, that is, by deleting all
of BoxMaker's elements from the chosen hyperedge $A_i$. Note that
the chosen box does {\em not} leave the game. At every point during
the game, and for every $1 \leq i \leq m$, we define \emph{the
weight} of box $A_i$ to be the number of BoxMaker's elements that
are currently in $A_i$, that is, the number of elements of $A_i$
that were claimed by BoxMaker and have not been deleted yet by
BoxBreaker.

\begin{theorem} [Theorem 2.3 in \cite{FHK}] \label{box1}
For every integer $k \geq 1$, BoxBreaker has a strategy for the game
$rBox(m,b)$ which ensures that, at any point during the first $k$
rounds of the game, every box $A_i$ has weight at most $b(1 +
\ln(m+k))$.
\end{theorem}

We will use this theorem as a strategy for Maker to obtain some minimum degree in his graph.
To that end, let $G=(V,E)$ be some graph, and let
$V_1,V_2 \subseteq V$ be arbitrary subsets.
By $(1:b) - \Deg (V_1,V_2)$ we denote the $(1:b)$ positional game
where Maker tries  to get
a large degree $d_M (v, V_2)$ for every $v \in V_1$.
Occasionally, we shall simply refer back to this as the {\em degree game}.
The following is an immediate conclusion of Theorem
\ref{box1}.

\begin{claim} [The degree game] \label{TheDegreeGame}
Let $G=(V,E)$ be a graph on $|V|=n$ vertices,
$V_1,V_2 \subseteq V$,
 and let $b$ be an
integer. Then, in the $(1:b) - \Deg (V_1,V_2)$ game, Maker can ensure
that $d_B(v, V_2) \leq 4b \ln n\, (d_M(v, V_2)+1)$ for every vertex $v\in V_1$.
\end{claim}

\begin{proof}
Maker pretends he is BoxBreaker and that he is playing the
$rBox(n,2b)$ game with the boxes $E(\{v\}, N(v, V_2))$, $v \in V_1$.
Notice that these boxes are not necessarily disjoint,
since we did not require $V_1$ and $V_2$ to be disjoint.
However, any edge belongs to
at most two of these boxes.
So BoxBreaker can pretend that the boxes are
disjoint and that BoxMaker claims $2b$ elements in every move
(using the Trick of fake moves). Now,
according to Theorem \ref{box1}, BoxBreaker can ensure that at any
point during the first $k$ rounds of the game, every box has weight
at most $2b(1 + \ln(n+k))$. Hence, at the end of the game every box
has weight at most $4b\ln n$. So, for every vertex $v \in V_1$, Maker
(BoxBreaker) has claimed at least one incident edge of $v$ for every
$4b\ln n$ incident edges Breaker (BoxMaker) has claimed.
Hence $d_B(v,V_2)
\leq 4b \ln n \, (d_M(v,V_2) +1)$.
\end{proof}

\subsection{General graph theory results}
\label{subsec::graphtheory} We will use the following graph which
was introduced in \cite{FH}. Let $k \geq 2$ and $n \geq 3(k-1)$ be
positive integers such that $(k-1) \mid n$. Let $m :=
\frac{n}{k-1}$. Let $C_1, \ldots, C_{k-1}$ be $k-1$ pairwise vertex
disjoint cycles, each of length $m$. For every $1 \leq i < j \leq
k-1$ let $P_{ij}$ be a perfect matching in the bipartite graph
$(V(C_i) \cup V(C_j), \{ uv : u \in V(C_i), \; v \in V(C_j)\})$.
Let ${\mathcal G}_k$ be the family of all graphs $G_k = (V_k, E_k)$
where $V_k = \bigcup_{i=1}^{k-1} V(C_i)$ and $E_k =
\left(\bigcup_{i=1}^{k-1} E(C_i)\right) \cup \left(\bigcup_{1 \leq i
< j \leq k-1} P_{ij}\right)$.

For the convenience of the reader we prove the following lemma.
\begin{lemma} \label{lem::kconnected}
For all integers $k \geq 2$ and $n \geq 3(k-1)$ such that $(k-1)
\mid n$, every $G_k \in {\mathcal G}_k$ is $k$-regular and
$k$-vertex-connected.
\end{lemma}

\begin{proof}
For $k=2$, the lemma is trivial. So assume $k\geq 3$. It is obvious
that $G_k$ is $k$-regular. Let $S \subseteq V_k$ be an arbitrary set
of size at most $k-1$. We will prove that $G_k \setminus S$ is
connected. Assume first that there exists some $1 \leq i \leq k-1$
for which $S \cap V(C_i) = \emptyset$. Then $(G_k \setminus S) \cap
C_i = C_i$ is connected and $V(C_i)$ is a dominating set of $G_k
\setminus S$. It follows that $G_k \setminus S$ is connected. Assume
then that $|S \cap V(C_i)| = 1$ for every $1 \leq i \leq k-1$.
Hence, $(G_k \setminus S) \cap C_i$ is a path on $k-1$ vertices for
every $1 \leq i \leq k-1$. Since $k-1 \geq 2$ and $|S \cap V(C_i)| =
1$ for every $1 \leq i \leq k-1$ hold by the assumption, it follows
that there is at least one edge between $(G_k \setminus S) \cap C_i$
and $(G_k \setminus S) \cap C_j$ for every $1 \leq i < j \leq k-1$.
It follows that $G_k \setminus S$ is connected.
\end{proof}

The following lemma shows that if a directed graph satisfies some
pseudo-random properties then it contains a long directed path. We
will use it in the proof of Theorem \ref{HAMgame}.

\begin{lemma} [Lemma 4.4 \cite{BKS}] \label{l:LongDirectedPath}
Let $m$ be an integer and let $D=(V,E)$ be an oriented graph with
the following property: There exists an edge from $S$ to $T$ between any two disjoint
sets $S,T\subseteq V$ such that $|S|=|T|=m$.
Then, $D$ contains a directed path of length at least
$|V|- 2m+1$.
\end{lemma}

The next lemma provides a sufficient (Hall-type) condition for a bipartite graph
to contain a perfect matching.

\begin{lemma} \label{BipPerMatch}
Let $G=(U_1\cup U_2,E)$ be a bipartite graph with $|U_1|=|U_2|=n$.
Let $r\leq n/2$ be an integer such that:
\begin{description}
\item [$(B1)$] For $i\in \{1,2\}$ and every $X\subset U_i$ of size $|X|\leq r$,
            $|N(X)|\geq |X|$.
\item [$(B2)$] For every $X\subset U_1$ and $Y\subset U_2$ with $|X|=|Y|=r$, $|E(X,Y)|>0$.
\end{description}
Then $G$ has a perfect matching.
\end{lemma}

\begin{proof}
In order to prove that $G$ admits a perfect matching we will prove
that $G$ satisfies Hall's condition, that is,  $|N(X)|\geq X$ for
every $X\subseteq U_1$ (see e.g. \cite{West}).
We show that for every $i\in \{1,2\}$ and for every  $X \subseteq U_i$ of
size $|X|\leq n/2$ we have that $|N(X)|\geq |X|$.
This readily implies Hall's condition.
We distinguish two cases:

\textit{Case 1:} If $|X|\leq r$, then by $(i)$ we have that
$|N(X)|\geq |X|$ and we are done.

 \textit{Case 2:} If $r<|X|\leq
n/2$. By $(ii)$ we have that $|N(X)|\geq n-r \geq n/2 \geq |X|$.
\end{proof}

\subsection{Properties of \gnp}
\label{subsec::gnp}

This subsection specifies properties (A1)-(A3) that a graph $G \sim \gnp$
fulfils a.a.s. It turns out that these properties are all we need to
prove our main theorems. So in fact, we could strengthen them to hold for
any graph $G$ that has suitable pseudo-random properties.

\begin{lemma} \label{l:propertiesofgnp}
Let ${K}\geq 2$ and let $G\sim \gnp$ with $p=\ln^K n/n$. Further,
let $\alpha \in \reals$ such that $1 \leq \alpha <{K}$, and let
$f=f(n)$ be some function that satisfies $1\leq f = O((\ln\ln n)^3)$.
Then, a.a.s.
\begin{description}
\item [(A1)] $\delta(G)=\Theta\left(\ln^K n\right)$ and $\Delta(G)=\Theta\left(\ln^K n\right)$.
\item [(A2)] For every subset $U\subseteq V$, $|E(U)|\leq \max\{100\, |U|\ln n,\: 100\,|U|^2p\}$.
\item [(A3)] For any two disjoint subsets $U,W\subseteq V$ with $|U|=|W|= n f^{-1} \ln^{-\alpha}n$,
                $|E(U,W)|=\Omega \left(n f^{-2} \ln^{{K}-2\alpha}n\right)$ and
                $|E(U)|=\Omega \left(n f^{-2} \ln^{{K}-2\alpha}n \right).$
\end{description}
\end{lemma}

\begin{proof}
To prove (A1), let $v\in V$. Since $d_G(v)\sim Bin(n-1,p)$ we
conclude $\Exp (d_G(v))=(n-1)p= \ln^K n(1-o(1))$. Hence, by Lemma \ref{Che},
$$\Prob \Big(d_G(v) \leq (1-1/2)\ln^K n\Big) \leq \exp \left(-\frac{\ln^K n}{8} (1-o(1)) \right)=o(1/n).$$
Now, by the union bound argument we conclude that
$$\Prob \Big(\exists \: v \in V: \:  d_G(v)\leq \frac{\ln^K n}{2}\Big) \leq n\cdot o(1/n)=o(1).$$
Similarly, by Lemma \ref{Che}, we obtain
$$\Prob \Big( \exists \: v \in V: \: d_G(v) \geq 2\ln^K n\Big)=o(1).$$
To prove (A2), let $U \subseteq V$ be a fixed subset of size $t := |U|$.
Then
\[ \Exp (|E(U)|) \leq t^2 p < 100 t^2 p \leq \max \{100 t \ln n, 100 t^2p\}. \]
Thus, by Lemma \ref{l:Che},
\begin{align*}
    \Prob \Big( |E(U)| \geq \max \{100 t\ln n, 100 t^2p\} \Big)
    &\leq \exp \Big( - \max \{100 t \ln n, 100 t^2p\} \Big) \\
    &\leq \exp \Big( - 100 t \ln n \Big).
\end{align*}
It follows that
\begin{align*}
    &\Prob \Big( \exists \ U \subseteq V \ : \ |E(U)| > \max \{100|U|\ln n, 100|U|^2p\} \Big) \\
    &\qquad \leq \sum_{t=1}^n \binom{n}{t} \exp ( - 100 t \ln n ) \\
    &\qquad \leq \sum_{t=1}^n  \exp \Big(t \ln n - 100 t \ln n \Big) \\
    &\qquad \leq \sum_{t=1}^n  \exp \Big(- 99 \ln n \Big) \\
    &\qquad =  \exp ( - 98 \ln n ) = o(1). \\
\end{align*}
For (A3), let $U,W\subseteq V$ be two disjoint subsets such that
$|U|=|W|= n f^{-1} \ln^{-\alpha}n$. Note that $|E(U,W)|$ is
binomially distributed with expectation $\mu_n := n
f^{-2}\ln^{{K}-2\alpha}n$. So by Lemma \ref{Che} we have that $\Prob
\left( |E(U,W)| \leq \frac{\mu_n}{2} \right) \leq \exp
\left(-\frac{\mu_n}{8}\right).$ Applying union bound we get that
\begin{align*}
    &\Prob \left(\exists \textrm{ disjoint } U,W \subseteq V \ : \  |U|=|W|=\frac{n}{f \ln^\alpha n} \textrm{ and } |E(U,W)| \leq \frac{\mu_n}{2} \right) \\
    &\qquad \leq \binom{n}{\frac{n}{\ln^\alpha n}}^2 \exp \left( -\frac{\mu_n}{8} \right) \\
    &\qquad \leq \exp \Bigg( \frac{2n}{\ln^\alpha n}(\alpha \ln \ln n +1)
        - \frac{n\ln^{{K}-2\alpha}n}{8f^2}\Bigg) \\
    &\qquad = o(1),
\end{align*}
since $K > \alpha$. Finally, since $|E(U)|\sim
Bin\left(\binom{|U|}{2},p\right)$, it follows analogously that
a.a.s. $|E(U)|=\Omega \left(n f^{-2} \ln^{{K}-2\alpha}n\right)$ for
all $U$ with $|U|= n f^{-1} \ln^{-\alpha}n$.
\end{proof}

\subsection{Expanders}
\label{subsec::expanders}

\begin{definition}
Let $G=(V,E)$ be a graph with $|V|=n$. Let $R:=R(n)$ and $c:=c(n)$
be two positive integers. We say that the graph $G$ is an
\emph{$(R,c)$-expander} if it satisfies the following two
properties:

\begin{description}
\item [$(E1)$] For every subset $X\subseteq V$
        with $|X|\leq R$, $|N(X)\setminus X|\geq c|X|$.
\item [$(E2)$] $|E(X,Y)|>0$ for every two disjoint subsets $X,Y\subseteq V$ of size
$|X|=|Y|=R$.
\end{description}
\end{definition}

Recall that a graph $G=(V,E)$ is called \emph{Hamilton-connected} if for every $x,y\in V$, the graph $G$ contains a Hamilton path with $x$ and $y$ as its endpoints.
The following sufficient condition for a graph to be Hamilton-connected was introduced in \cite{HKS}.

\begin{theorem} \label{HamCon}
Let $n$ be sufficiently large, and let $G=(V,E)$ be an $\Big(n/\ln
n,\, \ln \ln n\Big)$-expander on $n$ vertices. Then $G$ is
Hamilton-connected.
\end{theorem}

That is, by ensuring expander properties (locally), we can enforce a Hamilton cycle in the
whole graph (global property).

The following theorem lies in the heart of all of our proofs. It
says that in a subgraph of $G$ of sublinear order where certain
properties hold Maker is able to build a suitable expander fast,
that is in $o(n)$ moves.

\begin{theorem} \label{th::FastExpGame}
Let $b$ be an integer, ${K}> 12$, let $n$ be a sufficiently large
integer and let $p=\ln^K n/n$. Let $H=(V_H,E_H)$ be a graph on
$|V_H|=\Theta\left(n/\ln^{4}n\right)$ vertices and let $M$ and $F$
be two edge disjoint subgraphs of $H$, where $E_M$ already belongs
to Maker and $F$ consists of free edges. Assume that the following
properties hold:
\begin{enumerate}[(1)]
\item \label{th::FastExpGame1} $E_M \cup E_F=E_H$.
\item \label{th::FastExpGame2} There exist a partition $V_H = A_1 \cup A_2$ and
    a constant $c_1>0$ such that
    \begin{align*}
        A_1 &=\{v\in V_H:d_M(v)\geq c_1\ln^{{K}-6}n\}, \\
        A_2 &=\{v\in V_H:d_F(v)\geq c_1\ln^{{K}-4}n\}, \text{ and} \\
        |A_1|&= O(n\cdot \ln^{6-{K}}n).
    \end{align*}
\item \label{th::FastExpGame3} For any two disjoint subsets $U,W\subseteq V_H$ of size
        $|U|=|W|=\frac{n}{\ln^5n\, (\ln\ln n)^3}$, \\
        $|E_H(U,W)|=\Omega\left(\frac{n\ln^{{K}-10}n}{(\ln\ln n)^6}\right)$.
\item \label{th::FastExpGame4}For every subset $U\subseteq V_H$, $|E_H(U)|\leq \max\{100 \, |U|\ln n,\: 100\,|U|^2p\}$.
\end{enumerate}
Then, for every $c\leq \ln \ln (|V_H|)$ and $R=|V_H|/\ln (|V_H|)$,
in the $(1:b)$ Maker-Breaker game played on $E_H$, Maker has a
strategy to build an $(R,c)$-expander within $o(n)$ moves.
\end{theorem}

Before we prove this theorem we need an auxiliary result. Consider a
graph $H$ with (edge-disjoint) subgraphs $M$ and $F$ such that the
assumptions of Theorem \ref{th::FastExpGame} hold. Given a subgraph
$H_1=(V_H,E_1)$ of $H$, we denote $M_1$ and $F_1$ to be the
restrictions of $M$ and $F$ respectively to the subgraph $H_1$. The
following lemma says that in $H$ we can find a sparse subgraph with
suitable properties that will guarantee Maker's win in the expander
game.

\begin{lemma}\label{l:sparsergraph}
Under the assumptions of Theorem \ref{th::FastExpGame}
there exists a subgraph $H_1=(V_H,E_1)$ of $H$ with the
following properties:
\begin{enumerate}[(i)]
\item For every $v\in A_2$, $d_{F_1}(v)=\Omega(\ln^{3}n)$.
\item \label{l:sparsergraphii} For any two disjoint subsets $U,W\subseteq V_H$ such that $|U|=|W|=\frac{n}{\ln^5n\, (\ln\ln n)^3}$, \\
    $|E_1(U,W)|=\Omega\left(\frac{ n}{\ln^{3}n \, (\ln\ln n)^6} \right)$ .
\item For every $U\subseteq V_H$, $|E_1(U)|\leq \max\{1000|U|\ln n,1000|U|^2\ln^7n/n\}$.
\item \label{l:sparsergraphiv} $|E_1|=o(n)$.
\end{enumerate}
\end{lemma}

Note that all size parameters in this lemma do not depend on $K$
anymore. We want to stress that this is crucial for obtaining $|E_1|
=o(n)$.

\begin{proof}
Let $\rho= \ln ^{7-{K}}(n)$. Pick every edge of $H$ to be an edge of
$H_1$ with probability $\rho$ independently of all other choices.
Let $s(n):=\frac{n}{\ln^5n\, (\ln\ln n)^3}$. The properties
$(i)$-$(iv)$ will all be proven by identifying the correct binomial
distribution and by applying Chernoff- and union-bound-type
arguments.

To prove property $(i)$, notice that for every $v\in A_2$ the degree
of $v$ in $F_1$ is binomially distributed, that is, $d_{F_1}(v)\sim
Bin(d_F(v),\rho)$ with mean $\Exp (d_{F_1}) \geq c_1 \ln^3 n$.

Therefore, by Lemma \ref{Che} we have $\Prob\Big(d_{F_1}(v)\leq
\frac{1}{2}\, c_1 \ln^3 n \Big) \leq \exp \Big(-\frac{1}{8}\,
c_1\ln^{3}n\Big)$.
Hence, by the union bound we conclude $\Prob\Big(\exists\: v\in
A_2:\, d_{F_1}(v)\leq \frac{1}{2}\,c_1\ln^{3}n\Big)=o(1)$.

For property $(ii)$, let $c_2>0$ be such that $|E_H(U,W)|\geq
c_2n\ln^{{K}-10}n/(\ln\ln n)^6$ for every two disjoint subsets
$U,W\subseteq V_H$ with $|U|=|W|=s(n)$. This $c_2$ clearly exists by
assumption \eqref{th::FastExpGame3}.
Let $U,W$ be such subsets. Since $|E_1(U,W)|\sim Bin(E_H(U,W),\rho)$
with mean $\Exp (|E_1(U,W)|) \geq \frac{c_2 n}{\ln^{3}n\, (\ln\ln
n)^6}$, by Lemma \ref{Che},
\begin{align*}  \Prob \Bigg(|E_1(U,W)| \leq \frac{c_2 n}{2 \ln^{3}n\,(\ln\ln n)^6}\Bigg)
      &\leq \exp \left(-\frac{c_2n}{8\ln^{3}n\, (\ln\ln n)^6}\right).
\end{align*}
Therefore, by union bound we conclude that
\begin{align*}
    & \Prob\Bigg(\exists \: U,W\subseteq V_H : \: |U|=|W|=s(n) \textrm{ and }|E_1(U,W)| \leq \frac{c_2n}{2 \ln^{3}n\,(\ln\ln n)^6}\Bigg)\\
        & \qquad \leq \binom{n}{\frac{n}{\ln^5n}}^2
                \exp \Bigg(-\frac{c_2n}{8 \ln^{3}n\, (\ln\ln n)^6}\Bigg)\\
        & \qquad \leq \exp \Bigg(\frac{2n}{\ln^4 n} -\frac{c_2n}{8 \ln^{3}n\, (\ln\ln n)^6}\Bigg)
        =o(1).
\end{align*}

To prove property $(iii)$, note that  $|E_1(U)|\sim
Bin(E_H(U),\rho)$ with expectation $\Exp (|E_1(U)|) \leq \max
\{100\, |U|\ln n,\: 100\,|U|^2\ln^7n/n \}$. Again, by Lemma
\ref{l:Che} and union bound we get that:
\begin{align*}
    & \Prob \Big(\exists\: U\subseteq V_H : \: |E_1(U)|\geq
        \max\Big\{1000\, |U|\ln n,\: 1000\,|U|^2\ln^7n/n\Big\}\Big)\\
    & \qquad \leq \sum_{t=1}^{|V_H|} \binom{|V_H|}{t}\exp\Big(-\max\Big\{1000t\ln n,\:1000t^2\ln^7n/n\Big\}\Big)\\
    &\qquad \leq \sum_{t=1}^{n} \exp\Big(t\ln n - 1000t\ln n\Big)\\
    & \qquad \leq n\exp(-999\ln n)=o(1).
\end{align*}

For property $(iv)$, notice that $|E_1|\sim Bin(|E_H|,\rho)$. By
condition \eqref{th::FastExpGame4}, and since $|V_H| = \Theta
(n/\ln^4n)$, $|E_H|=O(n\ln^{{K}-8}n)$.
 So the expected size of $E_1$ is
$\mu=O(n\ln^{{K}-8}n\rho)=o(n)$.
Hence, again, by Lemma \ref{Che} we conclude that $|E_1|=o(n)$ with
probability tending to 1.

We have shown that in the randomly chosen subgraph the properties
$(i) - (iv)$ hold a.a.s. In particular, there exists an instance where all hold.
\end{proof}

Now, we are ready to prove Theorem \ref{th::FastExpGame}.

\begin{proof} [of Theorem \ref{th::FastExpGame}]

Let $H_1=(V_H,E_1)$ be a subgraph of $H$ as given by Lemma
\ref{l:sparsergraph}. For achieving his goal, Maker will play two
games in parallel on $E_1$. In the odd moves Maker plays the $(1:2b)$
degree game on $F_1$ and in the even moves he plays as $\cf$-Breaker the
$(2b:1)$ game $(E_1,\cf)$, where the winning sets are
\[ \cf=\Bigg\{ E_1(U,W) : U,W\subseteq V_H,\, U\cap W=\emptyset \textrm{ and }
    |U|=|W|=\frac{n}{\ln^5n\, (\ln \ln n)^3}\Bigg\}.\]
Combining Claim \ref{TheDegreeGame} and Lemma \ref{l:sparsergraph},
Maker can ensure with his odd moves that
\begin{equation} \label{degA2}
\text{for every } v\in A_2 \ : d_{M\cap H_1}(v)=\Omega (\ln^2n).
\end{equation}
Also, by Lemma \ref{l:sparsergraph}~\eqref{l:sparsergraphii},
\begin{align*}
    \sum_{F\in \cf}2^{-|F|/2b}
        & \leq\binom{n}{\frac{n}{\ln^5n}}^2
        2^{-\Omega\Big(n/(\ln^{3}n\, (\ln \ln n)^6)\Big)}\\
        & \leq \exp\Bigg(\frac{2n}{\ln^4n}
            - \Omega\left(\frac{n}{\ln^{3}n\, (\ln \ln n)^6}\right)\Bigg) =o(1).
\end{align*}

So by Theorem \ref{bwin} Maker (as $\cf$-Breaker) wins the game
$(E_1,\cf)$. That is, for any two disjoint subsets $U,W \subseteq
V_H$ of size $|U|=|W|=\frac{|V_H|}{\ln
(|V_H|)}=\Theta\left(\frac{n}{ \ln^5n}\right)=\omega(\frac{n}{\ln^5n\,
(\ln \ln n)^3})$, Maker can claim an edge between $U$ and $W$. Note
that this gives condition $(E2)$ of the expander definition, with $R
= |V_H|/ \ln (|V_H|)$. Furthermore, by Lemma
\ref{l:sparsergraph}~\eqref{l:sparsergraphiv}, the game lasts
$|E_1|=o(n)$ moves.

To prove that by the end of this game Maker's graph is indeed a
$\Big(|V_H|/ \ln (|V_H|),\: \ln\ln(|V_H|)\Big)$-expander, it remains
to check condition (E1).

Assume for a contradiction that there exists a set
$X\subseteq V_H$ such that
\begin{align} \label{stars1}
    &|X|\leq |V_H|/ \ln (|V_H|) \quad \text{ and } \quad |X\cup N_{M}(X)|\leq 2\ln \ln(|V_H|)|X|.
\end{align}
We distinguish three cases.

\textbf{Case 1:} $|X\cap A_1|\geq |X|/2$. Then $|X|=O(n \cdot
\ln^{6-{K}}n)$ by assumption \eqref{th::FastExpGame2}. Hence, and by
assumption \eqref{th::FastExpGame4} and \eqref{stars1},
\begin{align*}
    |E_H(X,N_M(X))| & \leq |E_H(X\cup N_M(X))|\\
        &\leq \max\Big\{100\,|X\cup N_M(X)|\,\ln n,\: 100\,|X\cup N_M(X)|^2p\Big\}\\
        & = O \Big(\max\Big\{|X|\ln \ln (|V_H|) \ln n, |X| \ln^2 \ln (|V_H|)\, \ln^{6}n\Big\}\Big).
\end{align*}
But this implies $|E_H(X,N_M(X))| = o(|X| \ln^{{K}-6}n)$ since
${K}>12$. However, since every vertex $v\in A_1$ has Maker degree at
least $c_1 \cdot \ln^{{K}-6}n$ we also conclude that
$|E_M(X,N_{M}(X))|=\Omega(\ln^{{K}-6}n\cdot |X|)$, a contradiction.

\textbf{Case 2:} $|X\cap A_2|\geq |X|/2$ and
$|X|<\frac{n}{\ln^5n\, (\ln \ln n)^3}$. By \eqref{degA2}, for every $v\in
X\cap A_2$, $d_{M\cap H_1}(v)=\Omega(\ln^2n)$.
Hence, $|E_{M\cap H_1}(X,N_M(X))|=\Omega(|X|\ln^2n)$.\\
On the other hand, by Lemma \ref{l:sparsergraph},
\begin{align*}
    |E_{H_1}(X,N_M(X))|
        &\leq \max\Big\{1000\, |X\cup N_M(X)|\, \ln n,\: 1000\,|X\cup N_M(X)|^2\,\ln^7n/n\Big\}\\
        &= O\Big(\max\Big\{|X|\ln n\ln\ln(|V_H|),|X|^2\ln^2\ln(|V_H|)\, \ln^{7}n/n\Big\}\Big)\\
        &=o(|X|\ln^2n),
\end{align*}
where the first equality follows from \eqref{stars1}. But this, again, is a contradiction.

\textbf{Case 3:} $\frac{n}{\ln^5n\, (\ln\ln n)^3}\leq |X| \leq
\frac{|V_H|}{\ln (|V_H|)}$. Since Maker wins (as $\cf$-Breaker) the
game $(E_1,\cf)$ we conclude that
\[ |N_M(X)|\geq |V_H|-\frac{n}{\ln^5n\, (\ln\ln n)^3}
        =\Omega\left(\frac{n}{\ln^4n}\right)=\omega(\ln\ln n|X|), \]
which contradicts \eqref{stars1}.
This completes the proof.
\end{proof}
\section{The Perfect Matching Game}
\label{sec::PMgame}
In this section we prove Theorem \ref{PMgame} and a variant for random bipartite graphs.

\begin{proof} [of Theorem \ref{PMgame}]
First we describe a strategy for Maker and then we prove it is a
winning strategy. At any point during the game, if Maker cannot
follow the proposed strategy (including the time limits) then he
forfeits the game. Before the game starts, Maker picks a
subset $U_0 \subseteq V$ of size $|U_0| = \frac{n}{\ln^4n}$
such that for every $v\in V$, $d(v,U_0)= \Omega(\ln^{{K}-4}n)$.
Such a subset exists because a randomly chosen subset of size
$\frac{n}{\ln^4n}$ has this property by
a Chernoff-type argument a.a.s.
Now, we divide Maker's strategy into two main stages.

\textbf{Stage I:} At this stage, Maker builds a matching  $M_0$ of
size $n/2-n/\ln^4n$ which does not touch $U_0$. Moreover, Maker
wants to ensure that by the end of this stage, for every $v\in V$,
\begin{align}
\label{degreeMFcondition}
    d_F(v,U_0)&=\Omega \left(\ln^{{K}-4}n\right),
    \quad \text{ or } \quad
    d_M(v,U_0)= \Omega \left( \ln^{{K}-6}n\right).
\end{align}
Initially, set $M_0=\emptyset$. For $i\leq n$, as long as
$|M_0|<n/2-n/\ln^4n$, Maker plays his $i$-th move as follows:
\begin{description}
\item [$(1)$] If there exists an integer $j$ such that $i=j \lfloor
\ln n \rfloor$, then Maker plays the degree game $(1:b\, \ln n) - \Deg (V,U_0)$.
\item[$(2)$] Otherwise, Maker claims an arbitrary free edge $e_i\in E$ s.t.
$e_i\cap e=\emptyset$ for every $e\in M_0$ and $e_i\cap
U_0=\emptyset$. Then, Maker updates $M_0$ to $M_0\cup \{e_i\}$.
\end{description}
When Stage I is over, i.e. $|M_0|= n/2-n/\ln^4n$, Maker proceeds to
Stage II.

\textbf{Stage II:} Let $V_H=V\setminus V(M_0)$ such that
$|V_H|=2n/\ln^4n$, and let $H:=(G-B)[V_H]$. We will show that $H$
together with the subgraphs $M$ consisting of Maker's edges and $F$
consisting of the free edges satisfies the conditions of Theorem
\ref{th::FastExpGame}. That is, Maker can play on $H$ according to
the strategy suggested by the theorem and build a suitable expander
in $o(n)$ moves.

Indeed, stage I and stage II constitute a winning strategy, i.e. if
Maker can follow the proposed strategy, he will get a perfect
matching of $G$. By Theorem \ref{th::FastExpGame}, Maker's subgraph
of $H$ will be an $(R,c)$-expander with $R=|V_H|/ \ln (|V_H|)$ and
$c=\ln\ln(|V_H|)$, for large $n$. By Theorem \ref{HamCon}, this
subgraph will be Hamilton-connected and that is why it will contain
a perfect matching $M_1$. Together with $M_0$ this forms a perfect
matching of $G$. Furthermore, Maker will win in $n/2+o(n)$ moves,
since Stage I lasts at most $n/2+o(n)$ rounds, whereas in Stage II
Maker needs only $o(n)$ moves. Thus, we only need to guarantee that
Maker can follow the strategy.

By Lemma \ref{l:propertiesofgnp}, the properties (A1), (A2) and (A3)
hold a.a.s. for $G$. We condition on these, and henceforth assume
that $G$ satisfies (A1), (A2) and (A3), where $f \in \{ 1, (\ln \ln n)^3 \}$.
We consider each stage separately.

\textbf{Stage I:} First, consider part (2), that is when Maker tries
to build the matching $M_0$ greedily. Assume that Maker has to play
his $i$-th move in Stage I and $i\neq j \lfloor \ln n \rfloor$ for
any $j \in \nat$. Furthermore, assume that still
$|M_0|<n/2-n/\ln^4n$. Let $T:=V\setminus(V(M_0)\cup U_0)$. Then
$|T|>n/\ln^4n$. Thus, by (A3) ($f=1$), $|E(T)|=\omega(n)$.
Since $i\leq n$, Maker and Breaker have claimed
$O(n)$ edges so far. In particular, Maker can find a free edge in
$T$ to be added to $M_0$. Thus, he can follow part (2) of Stage I.

Secondly, consider part (1). It is clear that Maker can play the
degree game. Thus, we only need to prove that the desired degree
condition \eqref{degreeMFcondition} will hold. We already know
that there exists a constant $c_1>0$ with $d(v,U_0)\geq
c_1\ln^{{K}-4}n$ for every $v\in V$. If at the end of Stage I
Breaker has $d_B(v,U_0)\leq 0.5c_1\ln^{{K}-4}n$ for some $v \in V$,
then \eqref{degreeMFcondition} holds trivially.
Thus, we can assume that
$d_B(v,U_0)\geq 0.5c_1\ln^{{K}-4}n$. In this case, Claim
\ref{TheDegreeGame} gives $d_M(v,U_0)\geq 0.1c_1\ln^{{K}-6}n/b$.

\textbf{Stage II:} We only need to check whether the conditions of
Theorem \ref{th::FastExpGame} hold for $H=(V_H,E(H))$.
Firstly, $|V_H| = 2n/\ln^4n$. Also,
condition \eqref{th::FastExpGame1} holds trivially by the definition of $H$.

For condition \eqref{th::FastExpGame2}, note that because of the
degree condition \eqref{degreeMFcondition} we can find a constant
$c_2$ such that $V_H = A_1 \cup A_2$, where $A_1 := \{v\in
V_H:d_M(v)\geq c_2\ln^{{K}-6}n\}$ and $A_2 = \{v\in V_H:d_F(v)\geq
c_2\ln^{{K}-4}n\}$. Since stage I took at most $n$ rounds,
$|A_1| = O(n \cdot \ln^{6-{K}}n)$.

Towards condition \eqref{th::FastExpGame3}, note that by (A3) ($f=(\ln \ln n)^3$)
for every disjoint $U,W\subseteq V$ of size $\frac{n}{\ln^5n\, (\ln\ln
n)^3}$, $|E(U,W)|=\Omega\Big(\frac{n\ln^{{K}-10}n}{(\ln\ln n)^6}
\Big)$. Since stage I took at most $n$ rounds, Breaker has claimed
$O(n)$ edges. Hence, in the reduced graph (where Breaker's edges are
deleted), property \eqref{th::FastExpGame3} is satisfied.

Condition \eqref{th::FastExpGame4} follows by (A2) and since $H\subseteq G$.
\end{proof}

In the light of Theorem \ref{KCONgame}, i.e. the $k$-connectivity game,
we would like to get a similar result for a random bipartite graph.
That is, for even $n$ we denote by ${\fam=2 B}_{n,p}$ a bipartite
graph with two vertex classes of size $n/2$, where every possible
edge is inserted with probability $p$. We show that Maker can win
the perfect matching game on ${\fam=2 B}_{n,p}$ fast. The main
difference to the proof of Theorem  \ref{PMgame} is that Maker will
not build an expander, but will rather fulfill the conditions of
Lemma \ref{BipPerMatch}.

\begin{theorem} \label{PMGameBip}
Let $b \geq 1$ be an integer, let $K >12$, $p=\frac{\ln^{K}
(n)}{n}$, and let $G\sim {\fam=2 B}_{n,p}$. Then a.a.s. Maker wins
the $(1:b)$ perfect matching game played on $G$ within
$\frac{n}{2}+o(n)$ moves.
\end{theorem}

\begin{proof}
The proof is analogous to the proof of Theorem \ref{PMgame},
so we just sketch it here.

For $G=(U_1\cup U_2,E(G))\sim {\fam=2 B}_{n,p}$, first choose a
subset $U_0\subset U_1\cup U_2$ such that $|U_0\cap U_1|=|U_0\cap
U_2|=\frac{n}{2\ln^4n}$, and $d(v,U_0)= \Omega(\ln^{{K}-4}n)$
for every $v\in U_1\cup U_2$.

Then Maker divides the game into two stages.

\textbf{Stage I:} Maker again builds greedily a matching $M_0$ of
size $n/2-n/\ln^4n$ which does not touch $U_0$. Furthermore, Maker
ensures that by the end of this stage for some $c_1>0$,
$d_F(v,U_0)\geq c_1\ln^{{K}-4}n$ or $d_M(v,U_0)\geq c_1
\ln^{{K}-6}n$ for every $v\in U_1\cup U_2$.

\textbf{Stage II:} Let $V_H=V\setminus V(M_0)$ with $|V_H\cap
U_i|=\frac{n}{\ln^4n}$ and let $H=(G-B)[V_H]$. Maker plays similarly
to the strategy given in the proof of Theorem \ref{th::FastExpGame}.
This time, he will not build an expander like before. But he will
ensure that after $o(n)$ rounds his subgraph of $H$ will satisfy
conditions $(B1)$ and $(B2)$ of Lemma \ref{BipPerMatch} with
$r=|V_H|/ \ln(|V_H|)$.

Similarly to Lemma \ref{l:sparsergraph}, we find a sparser
subgraph $H_1 \subseteq H$ with the analogue properties for bipartite graphs.
As in the proof of \ref{th::FastExpGame}, Maker plays in
every even move the $(2b:1)$ game $(E_1,\cf)$ as $\cf$-Breaker where
$E_1$ is the edge set of $H_1$, and where
$$\cf=\Bigg\{ E_1(U,W) : U\subseteq U_1,\, W\subseteq U_2 \textrm{ and } |U|=|W|=\frac{n}{\ln^5n\, (\ln \ln n)^3}\Bigg\}.$$
Winning this game, he will ensure (B2) with $r=|V_H|/ \ln(|V_H|)$.

To obtain (B1), Maker plays in each odd move the $(1:2b)$ degree game.
\end{proof}

\section{The Hamiltonicity Game}
\label{sec::HAMgame}

In this section we prove Theorem \ref{HAMgame}.

\begin{proof}
First we describe a strategy for Maker and then we prove it is a
winning strategy. At any point during the game, if Maker cannot
follow the proposed strategy (including the time limits) then he
forfeits the game.
As in the perfect matching game, Maker picks a
subset $U_0 \subseteq V$ of size $|U_0| = \frac{n}{10 \ln^4n}$
such that for every $v\in V$, $d(v,U_0)= \Omega(\ln^{{K}-4}n)$.

We divide Maker's strategy into the following four main stages.

\textbf{Stage I:} At this stage, Maker builds a matching  $M_0$ of
size $n/2-n/(9\ln^4n)$ which does not touch $U_0$. Moreover, Maker
wants to ensure that by the end of this stage
$d_F(v,U_0)=\Omega\left(\ln^{{K}-4}n\right)$ or $d_M(v,U_0)= \Omega
\left( \ln^{{K}-6}n\right)$ for every $v\in V$. As soon as this
stage is over, Maker proceeds to Stage II.

\textbf{Stage II:}
For a path $P$ let $End(P)$ denote the set of its endpoints.
Throughout this stage, Maker maintains a set
$M_1$ of vertex disjoint paths, a subset $M_2 \subseteq M_1$ and a
set $End:=\{v\in V: \exists P\in M_1\setminus M_2 \textrm{ such that
} v\in End(P)\}$. Initially, $M_1:=M_0$ and $M_2=\emptyset$. Let $r$
be the number of rounds Stage I lasted. For every $i > r$, Maker
will play his $i$-th move of this stage as follows:

\begin{description}
\item [$(1)$] If there exists an integer $j$ such that $i=j \lfloor
\ln n \rfloor$, then Maker plays the degree game $(1:b\, \ln n)-\Deg (V, U_0)$.

\item[$(2)$] Otherwise, Maker claims a free edge $xy$ between two
vertices from $End$ which are endpoints of two disjoint paths. Also,
Maker updates $M_1$ by replacing the two old paths merged through
$xy$ by a new one. Note that this new path is not deleted from
$M_1$. Maker also updates $End$ accordingly.

\item[$(3)$] If there is any path $P$ of length at least $10\ln^{{K}/3}n$ then
Maker updates $M_2:=M_2\cup \{P\}$.
\end{description}

This stage ends when $|M_1| = \lfloor n/\ln^{{K}/3}n \rfloor $.
Thus, Stage II lasts not more than $n/2 + o(n)$ rounds.
When this stage ends, Maker proceeds to Stage III.

\textbf{Stage III:} In this stage Maker ensures that his graph on
$V\setminus U_0$ will contain a path $P$ of length at least
$n-n/\ln^4n$. Moreover, Maker does so within $o(n)$ moves. Let $s$
be the number of rounds Stage I and II lasted. For every $i > s$,
Maker plays his $i$-th move of this stage as follows:

\begin{description}
\item [$(1)$] If there exists an integer $j$ such that $i=j \lfloor
\ln n \rfloor$, then Maker plays the degree game $(1:b\, \ln n)-\Deg (V, U_0)$.

\item[$(2)$] Otherwise, consider the paths in $M_1$ of length at least
$3 \ln ^{K /4}n$. Maker tries to connect these paths, not
necessarily through their endpoints, but through points close to
their ends. The full details of this partial game will be given in
the proof below.
\end{description}

\textbf{Stage IV:} Let $x,y$ be the endpoints of $P$,
the long path created in Stage III. Let $V_H=
\left(V \setminus V(P)\right) \cup \{x,y\}$. At this stage Maker
builds a Hamilton path on $(G-B)[V_H]$ with $x,y$ as its endpoints.
Moreover, Maker does so within $o(n)$ moves.

It is evident that if Maker can follow the proposed strategy then he
wins the Hamiltonicity game within $n+o(n)$ moves. It thus remains
to prove that indeed Maker can follow the proposed strategy without
forfeiting the game.

By Lemma \ref{l:propertiesofgnp}, the properties (A1), (A2) and (A3)
hold a.a.s. for $G$. We condition on these, and henceforth assume
that $G$ satisfies (A1), (A2) and (A3), where $f \in \{ 1, (\ln \ln n)^3 \}$.
We consider each stage separately.

\textbf{Stage I:} The proof that Maker can follow the proposed
strategy for this stage is analogous to the proof that Maker can
follow Stage I of the proposed strategy in the proof of Theorem
\ref{PMgame}.

\textbf{Stage II:} Assume $i\neq j \lfloor \ln n \rfloor$. If
$|M_1|>n/\ln^{{K}/3}n$, then $|M_1\setminus M_2|=\Omega
\left(n/\ln^{{K}/3}n\right)$, since there can be at most
$n/(10\ln^{{K}/3}n)$ disjoint paths of length at least
$10\ln^{{K}/3}n$. Hence, $|End|=\Omega \left(n/\ln^{{K}/3}n\right)$
and by (A3) ($f=1$), we have that the number of edges of $G$ spanned
by $End$ is $\Omega \left(n\ln^{{K}/3}n\right)=\omega(n)$.
Therefore, we conclude that indeed Maker can claim a free edge in
$G[End]$.

\textbf{Stage III:} Let $U'=\{v \in V:\, v \text{ belongs to a path
of length } \leq 3\ln^{{K}/4}n \text{ in } M_1 \}$ and update
$M_1:=M_1 \setminus \{P:\, P\text{ is of length } \leq 3\ln^{{K}/4}n
\}$. Notice that
\[ |U'|\leq 3\ln^{{K}/4}n \cdot \frac{n}{\ln^{{K}/3}n}=o\left(\frac{n}{\ln^4n}\right).\]
So the sum of the lengths of all paths
in $M_1$ is at least
\begin{equation} \label{pathLengthSum}
    |V(M_0)\setminus U'|\geq n-n/4\ln^4n.
\end{equation}
For every path $P\in M_1$, define $L(P)$ and $R(P)$ to be the first
and last $\ln^{{K}/4}n$ vertices of $P$ (according to some fixed
orientation of the path). Notice that since $|V(P)|>3\ln^{{K}/4}n$
it follows that $L(P) \cap R(P)=\emptyset$ for every $P\in M_1$.
Now, let $m=n/\ln^{{K}/2}n$ and let ${\mathcal H}=(X,{\mathcal F})$
be the hypergraph whose vertices are all edges of $G- B$  with both
endpoints in $\bigcup_{P \in M_1}\left(L(P)\cup R(P)\right)$ and
whose hyperedges are:
$$\cf=\Bigg\{E_{G\setminus B}(S,T) :
\exists \textrm{ distinct } P_1,...,P_{2m} \in M_1 \textrm{ s.t }
S=\bigcup_{i=1}^m L(P_i),T=\bigcup_{i=m+1}^{2m}R(P_i)\Bigg\}.$$

Note that for $E_{G\setminus B}(S,T)\in\cf$,
$|S|=|T|=m\ln^{{K}/4}n=n/\ln^{{K}/4}n$ holds. Thus, by (A3) ($f=1$),
we have for an element of $\cf$ that
$$|E_{G\setminus B}(S,T)|\geq |E_G(S,T)|-bn=\Omega\left(n\ln^{{K}/2}n\right).$$
Moreover, by (A2), we get that
$|X|=O\left(\left(\ln^{{K}/4}n|M_1|\right)^2p\right)=O\left(n\ln^{5{K}/6}n\right).$

Now,
\begin{align*} \sum_{F \in \cf}2^{-|F|/\ln^{0.9{K}}n}
            & = \sum_{F \in \cf}2^{-\Omega(n\ln^{-0.4{K}}n)}\\
            & \leq \binom{|M_1|}{m}^2 2^{-\Omega(n\ln^{-0.4{K}}n)}\\
            & \leq \left( e\ln^{{K}/6}n\right)^{2n/\ln^{{K}/2}n} 2^{-\Omega(n\ln^{-0.4{K}}n)}\\
            & \leq \exp\Bigg(\frac{2n}{\ln^{{K}/2}n}(1+{K}\ln\ln n/6)- \Omega\Big(\frac{n}{\ln^{0.4{K}}n}\Big)\Bigg)=o(1).
\end{align*}

Thus, by Theorem \ref{bwin} Maker as $\cf$-Breaker can win the
$(\ln^{0.9{K}}n,1)$ game $(X, \cf)$. Lemma \ref{lem::fakeMoves}
therefore tells us that Maker can claim at least one element in
every $F\in \cf$ within $1+|X|/(\ln^{0.9{K}}n+1)=o(n)$ moves.

To complete Stage III, let us define the auxiliary directed graph
$D=(V_D,E_D)$ whose vertices are $\{P: P \in M_1\}$ and whose
directed edges are $\{(P,Q): E_M(R(P),L(Q)) \neq \emptyset\}$.
Notice that for every pair of disjoint subsets $S,T \subseteq V_D$
such that $|S|=|T|=m$, there exists an edge in $D$ from $S$ to $T$,
since Maker wins the game $(X, \cf)$. Now, we claim that Maker has a
path of the desired length in his graph.  By Lemma
\ref{l:LongDirectedPath}, $D$ contains a directed path ${\mathcal
P}=P_0\ldots P_t$ of length $t\geq |V_D|-2m+1$. Further, note that
any path in $M_1$ has length at most $20 \ln^{K / 3}n$. Combining
this with \eqref{pathLengthSum}, removing paths $P\in M_1$ which do
not appear in ${\mathcal P}$ and deleting unnecessary parts of
$L(P)$ and $R(P)$ from paths $P\in {\mathcal P}$ we conclude that
Maker has thus created a path of length at least
$n-n/4\ln^4n-2|M_1|\ln^{{K}/4}n-2m\cdot 20\ln^{{K}/3}n\geq
n-n/\ln^4n$.

\textbf{Stage IV:} Let $P$ be the long path Maker has created in
Stage III, and let $x,y$ be its endpoints. Denote $V_H=\left(V
\setminus V(P)\right)\cup \{x,y\}$. Analogously to the perfect
matching game, we can use Theorem \ref{th::FastExpGame} and Theorem
\ref{HamCon} on $H:=(G-B)[V_H]$.
That is, Maker can build an expander on a sparse subgraph, and thus
obtains a Hamilton
path in $H$ with $x,y$ as its endpoints in $o(n)$ moves. This
completes the proof.
\end{proof}

\section{The $k$-Connectivity Game}
\label{sec::KCONgame}

In this section we prove Theorem \ref{KCONgame}.
It is a simple application of the Hamiltonicity game, the Perfect-matching game
on random bipartite graphs, and the degree game.

\begin{proof}
Let $G \sim \gnp$, and randomly partition the vertex set into $k$
disjoint sets $V_1,...,V_{k-1},W$ where each $V_i$ has size
$\left\lfloor \frac{n}{k-1}\right\rfloor$ ($W$ might be empty).
For every $1\leq i \leq k-1$, let $G_i= G[V_i]$, and for every
$1\leq i < j \leq k-1$ let $G_{ij}$ be the bipartite subgraph of $G$
with parts $V_i$ and $V_j$. From the definition it is clear that
$G_i \sim \cG_{\left\lfloor \frac{n}{k-1}\right\rfloor,p}$ for every
$1\leq i \leq k-1$ and $G_{ij} \sim {\fam=2 B}_{2\left\lfloor
\frac{n}{k-1}\right\rfloor,p}$ for every $1\leq i<j\leq k-1$.

Now, Maker's strategy is to play the Hamiltonicity game on every
$G_i$, the perfect matching game on every $G_{ij}$, and for every
$w\in W$, Maker wants to claim $k$ distinct edges $ww'$
with $w' \in V\setminus W$
(recall that $G$ is typically such that $d(v)=
\Theta(\ln^K n)$ for every vertex  $v\in V(G)$).
Thus, in total
Maker plays on $t \leq k-1 + \binom{k-1}{2} + k-2 = \binom{k}{2}
+k \leq k^2$ boards. Enumerate all boards arbitrarily, and let Maker
play on board $i \mod t$ in his $i$-th move.
Between any two moves
on a particular board, Breaker has claimed at most $bk^2$ new edges
on this board. Using the trick of fake moves we can assume that
Maker plays the $(1:bk^2)$ Hamiltonicity game on every $G_i$, the
$(1:bk^2)$ perfect matching game on every $G_{ij}$, and the
degree-game $(1:bk^2)-\Deg(\{w\},V\setminus W)$ for every $w \in W$.
By Theorem \ref{HAMgame} and Theorem
\ref{PMGameBip}, every Hamiltonicity game and every perfect matching
game lasts $\left\lfloor \frac{n}{k-1}\right\rfloor +o(n)$ moves,
whereas the games $Deg(\{w\},V\setminus W)$ last
in total at most $k|W| = O(1)$ moves.
If Maker succeeds on some board (that is, either he formed a
Hamilton cycle on some $G_j$, or a perfect matching on some
$G_{j_1j_2}$, or $d_M(w,V\setminus W))\geq k$ for $w\in W$),
then he quits playing on that particular board. That is, he ignores
this board and plays on another one where he has not won yet.

By Lemma \ref{lem::kconnected} Maker is thus able to build a
$k$-connected graph on $G[V_1\cup \ldots \cup V_{k-1}]$. Also, since
for every $w \in W$, $\d_M(w, V \setminus W) \geq k$, Maker's final
graph will be $k$-connected.
In total, Maker plays at most
\begin{align*}
&(k-1) \left( \left\lfloor \frac{n}{k-1} \right\rfloor +o(n) \right)
 + \binom{k-1}{2} \left(\left\lfloor \frac{n}{k-1} \right\rfloor +o(n) \right)
 +O(1)  \leq \frac{kn}{2} +o(n)
\end{align*}
moves, as claimed.
\end{proof}

\section{Open problems}
\label{sec::OpenProblems}

We conclude with the list of several open problems directly relevant
to the results of this paper.

\textbf{Sparser graphs.} For the three games considered in this
paper, we would like to find fast winning strategies for Maker when
the games are played on  $G\sim \gnp$, where
$p=(1+\epsilon)\frac{\ln n}{n}$ for a constant $\epsilon>0$. Our
proofs heavily depend on the ability of Maker to build an expander
fast (cf. Theorem \ref{th::FastExpGame}), which does not seem
possible for such small $p$. We were not able to prove an analogue
to Lemma \ref{l:sparsergraph} for smaller $p$'s mainly because of
Property (iv) in this lemma. Therefore, we find it very interesting
to either find fast strategies for Maker substantially different
from ours, or alternatively provide Breaker with a strategy for
delaying Maker's win by a linear number of moves.

\textbf{Faster winning strategies for Maker.} In this paper we have
proved that Maker can win the perfect matching game, the
Hamiltonicity game and the $k$-connectivity game played on $G\sim
\gnp$ within $n/2+o(n)$, $n+o(n)$ and $kn/2+o(n)$ moves,
respectively. Although this is asymptotically tight it could be that
the error term does not depend on $n$. It would be interesting to
find the error term explicitly, or at least to provide tighter
estimates on it.

\textbf{Fast winning strategies for other games.} It would be very
interesting to prove similar results, i.e. fast winning strategies
for Maker, for other games played on $G\sim \gnp$. We suggest the
fixed-spanning-tree game. To be precise, let $\Delta \in \nat$ be
fixed, and let $(T_n)_{n \in \nat}$ be a sequence of trees on $n$
vertices with bounded maximum degree $\Delta (T_n) \leq \Delta$.
Maker's goal is to build a copy of $T_n$ within $n+o(n)$ moves.
Notice that this problem might be much harder than what we proved
since even the problem of embedding spanning trees into $G\sim \gnp$
is still not completely settled (for more details see, e.g \cite{K},
\cite{HKS1}).

\textbf{Winning strategies for Red.} The problems considered in this
paper were initially motivated by finding winning strategies for Red
in the strong games via fast winning strategies for Maker (see
\cite{FH}, \cite{FH1}). It would be very interesting to prove that
indeed typically Red can win the analogous strong games played on
$G\sim \gnp$.

\textbf{Acknowledgement:} The authors want to thank Peleg Michaeli
for fruitful conversations.

\end{document}